\documentclass{amsart}

\usepackage{amssymb,times} 

\usepackage[text={6.5in,9in},centering]{geometry}

\usepackage{amsrefs}

\usepackage{graphicx,color,enumerate}

\numberwithin{equation}{section}

\newtheorem{theorem}{Theorem}[section]
\newtheorem{proposition}[theorem]{Proposition}
\newtheorem{conjecture}[theorem]{Conjecture}
\newtheorem{lemma}[theorem]{Lemma}

\theoremstyle{definition}
\newtheorem{definition}[theorem]{Definition}
\newtheorem{algorithm}[theorem]{Algorithm}

\DeclareMathOperator{\Lie}{\textit{Lie}}
\DeclareMathOperator{\Grav}{\textit{Grav}}
\DeclareMathOperator{\Hycomm}{\textit{Hycomm}}
\DeclareMathOperator{\sgn}{sgn}
\DeclareMathOperator{\Aut}{Aut}
\DeclareMathOperator{\Ind}{Ind}

\usepackage[all,cmtip]{xy}

\newcommand{\V}{\mathcal{V}}
\renewcommand{\L}{\mathcal{L}}
\newcommand{\M}{\overline{M}}
\renewcommand{\H}{\overline{H}}
\renewcommand{\S}{\mathbb{S}}
\newcommand{\mg}[1]{\M_g^{\leqslant #1}}
\newcommand{\mog}[1]{\M_{0,2g+2}^{(#1)}}
\newcommand{\dmog}{\partial\mog{0}}
\newcommand{\hg}[1]{\H_g^{\leqslant #1}}
\newcommand{\C}{\mathbb{C}}
\newcommand{\Z}{\mathbb{Z}}
\renewcommand{\P}{\mathbb{P}}

\newcommand{\Set}[2]{\left\{ #1,\ldots , #2 \right\}}

\DeclareMathOperator{\asn}{asn}
\DeclareMathOperator{\ccd}{ccd}
\DeclareMathOperator{\ce}{ce}

\DeclareMathOperator{\Hur}{\textit{Hur}}
\DeclareMathOperator{\oHur}{\overline{\Hur}}

\title[Cohomological excess of moduli spaces]{The cohomological excess
  of certain moduli spaces of curves of genus $g$}

\author{Chitrabhanu Chaudhuri}

\address{Max Planck Institute for Mathematics, Bonn, Germany}
\email{chitro@mpim-bonn.mpg.de}

\begin{document}

\maketitle

\begin{abstract}
  The open set $\mg{k} \subset \M_g$ parametrizes stable curves of
  genus g having at most k rational components. By the work of Looijenga, 
  one expects that the cohomological excess of $\mg{k}$ is at most 
  $g-1+k$. In this paper we show that when $k=0$, the conjectured 
  upper bound is sharp by showing that there is a constructible sheaf 
  on $\hg{0}$ (the hyperelliptic locus) which has non-vanishing cohomology 
  in degree $3g-2$.
\end{abstract}

\maketitle

\section{Introduction}

This research originated from a conjecture by Looijenga which can be
stated as follows:
\begin{conjecture}
  The coarse moduli space $M_g$ of smooth curves of genus $g >1$ can 
  be covered by $g-1$ open affine subvarieties.
\end{conjecture}

Fontanari and Pascolutti \cite{FP} prove this conjecture for genus $2
\leq g \leq 5$. (The paper of Fontanari and Looijenga \cite{FL} is
also relevant.) This conjecture gives bounds on the topological
complexity of the moduli space. For example, it would imply the
cohomological dimension of constructible sheaves on $M_g$ is at most
\begin{equation*}
  \dim M_g + (g-2) = 4g-5 .
\end{equation*}
For local systems on $M_g$ this was established by Harer. Looijenga's
conjecture may be viewed as a generalization of Harer's theorem.

The minimum number of open affine subsets needed to cover a variety 
less one is called the \emph{affine covering number}. Later Roth and 
Vakil \cite{RV} introduced the closely related \emph{affine stratification 
number} ($\asn$). 

The Deligne-Mumford compactification $\M_g$ is a projective variety
which contains $M_g$ as the complement of a normal crossings divisor
and parametrizes stable curves of genus $g$, that is, projective curves
whose singularities are nodes and whose smooth locus has no components
of non-negative Euler characteristic.

Graber and Vakil \cite{GV} have introduced a filtration of this
variety: $\M_{g,n}^{\leqslant k}$ parametrizes stable curves of genus
$g$ having at most $k$ rational components. Roth and Vakil extend
Looijenga's conjecture to the following:
\begin{conjecture}
  \label{conj:asn}
  The affine stratification number of $\M_{g,n}^{\leqslant k}$ is at
  most $g-1+k$.
\end{conjecture}

Along with other topological consequences this conjecture would prove
that the cohomological dimension for constructible sheaves on
$\M_{g,n}^{\leqslant k}$ is at most $\dim \M_{g,n}^{\leqslant k} + g -
1 + k = 4g-4+n+k$. We shall denote the cohomological dimension for
constructible sheaves on a variety $X$ by $\ccd(X)$. It
is the minimum integer $d$ such that $H^n(X,\mathcal{F}) = 0$ for any
$n > d$ and any constructible sheaf $\mathcal{F}$ on $X$.

Looijenga introduced another invariant, called
\emph{cohomological excess} ($\ce$).
\begin{definition}
  The cohomological excess of a non-empty variety $X$, denoted $
  \ce(X)$, is the maximum of the integers $\ccd(W) - \dim W$, where
  $W$ runs over all the Zariski closed subsets $W \subset X$.
\end{definition}
Looijenga's aim was to give an upper bound for the cohomological 
excess of the moduli space of smooth curves $M_{g,n}$ and, more in 
general, for certain open subsets of $\M_{g,n}$. The expected upper
bound in the case of $\M_{g,n}^{\leqslant k}$ is the following:

\begin{conjecture} 
  \label{conj:ce}
  The cohomological excess of $\M_{g,n}^{\leqslant k}$ is at most
  $g-1+k$.
\end{conjecture}

The conjecture above is consistent with Conjecture~\ref{conj:asn},
since $\ce(X) \leq\asn(X)$. In this paper we show that the upperbound
in the conjecture above is sharp when $k =0$. We consider the locus
\begin{equation*}
  \hg{0} = \mg{0} \cap \H_g
\end{equation*}
of stable hyperelliptic curves in $\mg{0}$. Then there is a
constructible sheaf $\L$ on $\hg{0}$ whose cohomological dimension is
\begin{equation*}
  3g-2 = (g-1) + (2g-1) = (g-1) + \dim \hg{0} .  
\end{equation*}

The space $\H_g$ is a quotient of $\M_{0,2g+2}$ by the action of the
symmetric group $S_{2g+2}$. The constructible sheaf $\L$ on $\hg{0}$
is obtained by taking the push forward of the constant sheaf $\C$, 
under the quotient map. We prove the following result:

\begin{lemma}
  \label{lem:main}
  The cohomology group $H^{3g-2}(\hg{0},\L)$ is non-zero, and 
  $H^k(\hg{0},\L) =0$ for $k>3g-2$.
\end{lemma}

As a consequence we have:
\begin{theorem} \label{thm:upper}
  $\ce\left(\mg{0}\right) \geq g-1$
\end{theorem}

\noindent \emph{Remark.} By Conjecture~\ref{conj:ce} $\ce(\mg{0}) 
\leq g-1$, so then Theorem~\ref{thm:upper} shows that $\ce(\mg{0}) =
g-1$. It is not hard to see that $\asn(\hg{k}) \leq g-1+k$, hence
this result also shows that $\asn(\hg{0}) = g-1$.

\section{Combinatorial Preliminaries}

Here we recall some definitions from graph theory, explained in
greater detail in Getzler and Kapranov \cite[Section~2]{GK1}.

\begin{definition} 
  A \textbf{graph} $G$ is a tuple $(F,V,\sigma)$, where
  \begin{enumerate}
  \item $F(G)$ is the set of \textbf{flags} of $G$;
  \item $V(G)$ is a partition of $F(G)$, whose parts are called
    \textbf{vertices}.
  \item $\sigma: F(G) \to F(G)$ is an involution.
    \qedhere
  \end{enumerate}
\end{definition}

The fixed points of $\sigma$ are called \textbf{leaves} and the set of
all leaves is denoted by $L(G)$. The orbits of size $2$ of $\sigma$
are called \textbf{edges} and the set of all edges is denoted by
$E(G)$. Let
\begin{equation*}
  F(v) = \{ f \in F(G) |\ f \in v\}
\end{equation*}
be the set of flags incident on the vertex $v\in V(G)$.

A graph $G$ has a geometric realization $|G|$, which is the
one-dimensional cell complex with 1-cells $[0,1]\times F(G)$, modulo
the identifications $0\times f\sim 0\times f'$ if $f,f'\in F(v)$, and
$1\times f\sim1\times\sigma(f)$.

For example, the geometric realization of the graph
\begin{equation*}
  G = (\{1,\ldots,9\},\{\{1,4,6,8\},\{2,3,5,7,9\}\},(45)(67)(89))   
\end{equation*}
is shown in Figure~\ref{fig:GR}.

\begin{figure}[h]
\centering
\includegraphics[scale=1]{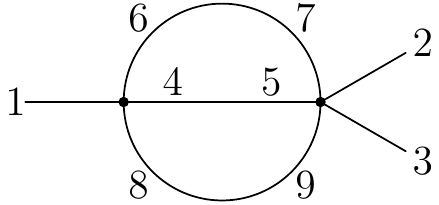}
\caption{The geometric realization of $G$}
 \label{fig:GR}
\end{figure}

Let $b_i(G)= \dim H_i(G,\C)$, for $i=0,1$. We only consider connected
graphs, that is graphs $G$ with $b_0(G)=1$. If $G$ is connected, the following
equality holds:
\begin{equation*}
  b_1(G) = |V(G)| - |E(G)| + 1 .
\end{equation*}
A \textbf{tree} is a graph $T$ with $b_0(T)=1$ and $b_1(T)=0$, that
is $|T|$ is connected and simply connected.

We consider graphs along with \textbf{labeling} $g:V(G) \to \Z_{\geq
  0}$ of the vertices. The number $g(v)$ is the genus of the vertex
$v$. The genus of the graph $G$ is
\begin{equation*}
  g(G) = \sum_{v\in V(G)} g(v) + b_1(G) .  
\end{equation*}

A graph $G$ is \textbf{stable} if its vertices satisfy the inequality:
\begin{equation*}
  2g(v)-2+|F(v)| > 0 .  
\end{equation*}
A graph of genus $g$ and $n$ leaves will be said to be of type $(g,n)$. 
For example, in Figure~\ref{fig:GR}, if both the vertices of the graph
have genus $0$, then the graph is stable of type $(2,3)$.

The \textbf{stabilization} of a labelled graph $G$ is constructed by
deleting vertices $v$ of $G$ of genus $0$ containing one or two flags.

A numbering of leaves is a bijection $L(G) \to [n]$, where
$[n]=\{1,\ldots,n\}$. A graph with a numbering of its leaves is called
a numbered graph.

Two numbered graphs are isomorphic if there is an isomorphism of the
underlying graphs that preserves the genus of vertices and the
numbering of leaves. Let $\Gamma(g,n)$ be the isomorphism classes
of stable graphs of type $(g,n)$. $\Gamma(g,n)$ is finite, as it has  
been proved in \cite[Lemma~2.16]{GK1}.

Let $M_{g,n}$ be the (coarse) moduli space of smooth genus $g$
algebraic curves over $\C$ with $n$ marked points, and let $\M_{g,n}$
be its Deligne-Mumford compactification, the moduli space of stable
curves of arithmetic genus $g$ with $n$ marked points.

To each stable curve of genus $g$ with $n$ marked points we associate
a stable graph of type $(g,n)$, called the \textbf{dual graph}: the
vertices of the dual graph correspond to the irreducible components of
the curve, each labelled by the geometric genus of the corresponding
component, the edges of the graph correspond to nodes, and the leaves
correspond to the marked points.

We have a stratification of $\M_{g,n}$ corresponding to isomorphism
classes of the dual graphs in $\Gamma(g,n)$, as explained in
\cites{GK1,V}. For an isomorphism class in $\Gamma(g,n)$ choose a
representative $G$. Let
\begin{equation*}
  M_G = \{ [C] \in \M_{g,n} \mid \text{ dual graph of } C \text{ is
    isomorphic to } G\}.
\end{equation*}
$M_G$ is open in its closure $\M_G$, and the set of subvarieties
$\{M_G\}$ as $G$ varies over isomorphism classes of graphs in
$\Gamma(g,n)$ stratifies $\M_{g,n}$.

There is a partial order $<$ on $\Gamma(g,n)$, $[G] < [G']$ if $G'$
can be obtained from a graph isomorphic to $G$ by contracting a subset
of the edges and relabelling the vertices (genus of a vertex of $G'$ is
the genus of the sub-graph which is the pre-image of the vertex, see
\cite[Section~2]{GK1}). Then $[G]<[G']$ if an only if $\M_G \subset
\M_{G'}$.

\section{Hyper-elliptic locus}
 
A hyperelliptic curve of genus $g$ is a smooth algebraic curve which
admits a degree 2 map to $\P^1$ ramified over $2g+2$ points. The locus
$H_g$ is the subvariety of $M_g$ parametrizing hyperelliptic curves of
genus $g$ and $\H_g$ is its closure in $\M_g$. There is an isomorphism:
\begin{equation*}
  \H_g \cong \M_{0,2g+2}/S_{2g+2}
\end{equation*}
Here, $S_n$ is the symmetric group on $n$ letters and it acts on
$\M_{0,n}$ by permuting the marked points.

Let us recall how the isomorphism is obtained. Let $\Hur_{g,d}$ be the
Hurwitz space parametrizing degree $d$ simply branched covers of
$\P^1$ of genus $g$. By simple branching we mean each fiber has at
least $d-1$ points. Let $\oHur_{g,d}$ be its compactification by
admissible covers, as in \cite[Section~4]{HM2}. By Riemann-Hurwitz
there are $r = 2g+2d-2$ points over which ramification occurs. We have
two maps; $\phi: \oHur_{g,d} \to \M_{0,r}/S_r$ by remembering only the
points in $\P^1$ over which branching occurs, and $\psi : \oHur_{g,d}
\to \M_g$. The admissible cover may not be a stable curve, but we can
stabilize it to obtain a genus $g$ curve and thus obtain $\psi$.

In case of $d=2$ and $r=2g+2$, $\phi$ is an isomorphism, and $\psi$ an
embedding onto $\H_g$, giving a very explicit description of
$\H_g$. Let us denote by $\pi: \M_{0,2g+2} \to \H_g$, the quotient map
to $\M_{0,2g+2}/S_{2g+2}$ followed by the isomorphism.

\subsection{Strata} 
\label{subsec:strata}

Associated to a curve $C$, such that $[C] \in \H_g$, are two combinatorial
objects. The first is the dual graph of $C$, which is a stable graph
of type $(g,0)$. The second is the dual graph of a curve $D$ such that
$[D] \in \pi^{-1}([C])$, which is a stable graph of type
$(0,2g+2)$. Of course if $[D'] \in \pi^{-1}([C])$ is another pre-image
then $[D'] = \sigma[D]$ for some permutation $\sigma \in
S_{2g+2}$. Hence, the dual graph of $D'$ is the same as the dual graph
of $D$ but with a renumbering of the leaves.

In fact, if we take the quotient $\Gamma(0,2g+2)/S_{2g+2}$, then the
graphs correspond to the strata of $\M_{0,2g+2}/S_{2g+2}$. These are
the graphs that we obtain if we forget the numbering of the
leaves. The following algorithm describes how to obtain the dual graph
of $\pi([D])$ from the dual graph of $D$. This defines a function
$\Gamma(0,2g+2) \to \Gamma(g,0)$ which as described factors through
$\Gamma(0,2g+2)/S_{2g+2}$, and in fact $\Gamma(0,2g+2)/S_{2g+2} \to
\Gamma(g,0)$ is injective as will be clear from the algorithm. But 
before giving the algorithm we have the following definitions:

Let $T=(F,V,\sigma)$ be a stable graph of type $(0,2k)$.  The
\textbf{parity} $p: F(T) \to \Z/2$ is a function satisfying
$p\circ\sigma = p$, so that both flags of an edge have same
parity. (We say that a flag is even or odd according to whether its
parity is $0$ or $1$: in drawing graphs, the even edges will be
dashed.) The leaves of $T$ are odd. The parity of an edge $e$ is
determined as follows: deleting $e$ produces two connected graphs
$G_1$ and $G_2$, both of which have either an even or an odd number of
leaves, since the total number of leaves must be $2k$; the edge $e$ is
even or odd accordingly.

\begin{definition}
  The \textbf{ramification number} $\rho(v)$ of a vertex $v \in V(T)$
  is the number of its flags that are odd
  \begin{equation*}
    \rho(v) = |\{f \in F(v) \mid p(f)=1\}| .
    \qedhere
  \end{equation*}
\end{definition}

\begin{algorithm}
  \label{Algo}
  Given a tree $T$ corresponding to a curve $C$ in $\M_{0,2g+2}$, the
  dual graph of $\pi([C])\in\M_g$ is the stabilization of the graph
  $G$ defined as follows:
  \begin{itemize}
  \item There are two edges in $G$ for each even edge of $T$, and one
    edge in $G$ for each odd edge of $T$. (The leaves of $T$ do not
    contribute flags to $G$.)
  \item A vertex $v$ of $T$ contributes a single vertex to $G$, of
    genus $(\rho(v)-2)/2$, unless $\rho(v)=0$, in which case it
    contributes two vertices of genus $0$.
    \qedhere
  \end{itemize}
\end{algorithm}

\begin{figure}[ht]
\centering
 \includegraphics[scale=1]{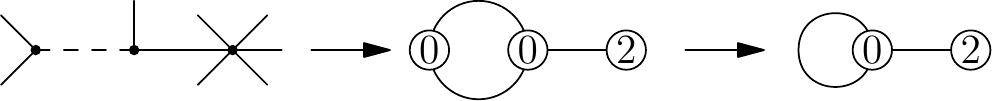}
 \includegraphics[scale=1]{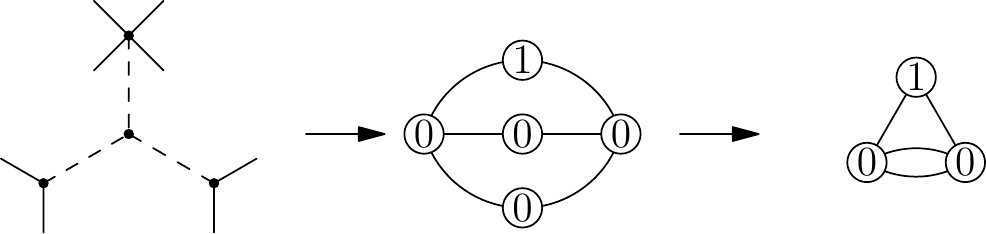}
\caption{Illustrations of the Algorithm}
\label{fig:Algorithm}
\end{figure}

Figure~\ref{fig:Algorithm} illustrates the algorithm. The left-most
graphs correspond to dual graphs of curves in $\M_{0,2g+2}$ (since all
vertices are genus 0 there is no need to label them). The middle
graphs are the dual graphs of the admissible covers and the rightmost
graphs are the graphs of the images in $\H_g$. Now let us see why the
above algorithm works. Most of this is in fact explained in
\cite[Section~3-G]{HM1}.

Let $C$ be a curve in $\M_{0,2g+2}$ and $f: C' \to C$ be the
admissible double covering. Assume that $C$ has 2 irreducible
components $C_1$ and $C_2$: the more general case is only notationally
more complicated.  If the node connecting the components is an odd
node, then both components have an odd number of marked points. Also
$f:f^{-1}(C_i) \to C_i$ are actual branched double covers, so by
Riemann-Hurwitz there must be even number of branch points and by
definition of admissible cover, there must be ramification over the
node. This shows that there is ramification over the odd nodes,
whereas similar reasoning shows that the even nodes have two pre-images
in the admissible cover.

It is easy to see that each $f^{-1}(C_i)$ is smooth and has Euler
characteristic $4-\rho(C_i)$, where $\rho$ is the branching number.
The rest is self-explanatory.

\subsection{A Filtration}

Graber and Vakil \cite{GV} have defined a filtration of $\M_g$ by open
subsets $\mg{k}$ corresponding to stable genus $g$ curves with at most
$k$ rational components (irreducible components of geometric genus
0). This filtration induces a filtration $\hg{k}$ on $\H_g$, and hence,
on $\M_{0,2g+2}$ through the map $\pi:\M_{0,2g+2} \to \H_g$. Let us
determine which strata of $\M_{0,2g+2}$ are in $\mog{k} :=
\pi^{-1}\left(\hg{k}\right)$.

For a tree $T$ of type $(0,2g+2)$, let the \textbf{edge-valence} of a 
vertex $v \in V(T)$ be the number $\nu(v) = |F(v)\setminus L(T)|$,
that is the number of edges of that vertex.

Call a vertex $v\in V(T)$ \textbf{internal} if $\nu(v)>1$; otherwise it 
is \textbf{external}.

\begin{proposition}
  \label{prop:admissible}
  Let $C$ be a curve in $\M_{0,2g+2}$ with corresponding dual graph
  $G$. Then the image $\pi([C]) \in \H_g$ has a rational component if
  and only if $G$ has an internal vertex $v$ with $\rho(v) \leq
  2$. Furthermore, the number of rational components of $\pi([C])$ is
  given by
  \begin{equation*}
    2\,|\{ v \in V(G) \mid \rho(v)=0 \}| +     |\{v \in V(G) \mid
    \textup{$v$ internal and $\rho(v) =2$} \}| .
    \qedhere
  \end{equation*}
\end{proposition}
\begin{proof}
  From Algorithm~\ref{Algo}, it is clear that the only vertices of $G$
  that contribute a genus $0$ vertex to the dual graph of the
  admissible cover are those vertices $v\in V(G)$ with $\rho(v) = 0$
  or $\rho(v) =2$.

  If $\rho(v) = 0$, then $v$ meets no leaves, so it is internal. Hence
  in the admissible cover it lifts up to $2$ vertices of genus $0$
  each of which is connected to at least $3$ edges, and hence survives
  stabilization. Hence $v$ contributes $2$ vertices of genus $0$ to
  the stabilization of the admissible cover.

  On the other hand if $\rho(v) = 2$, then the vertex lifts up to one
  vertex of genus $0$ in the admissible cover. If $\nu(v) =1$, then
  $v$ meets two leaves and an even edge. But then the vertex
  corresponding to it in the admissible cover has just 2 edges on it,
  and disappears after stabilization. If $\nu(v) >1$, then the
  corresponding vertex in the admissible cover meets at least $3$
  edges and survives stabilization.
\end{proof}

The proposition above tells us exactly which curves belong to
$\mog{k}$. The following bound will be useful later. For a stable
pointed curve $C$, let $\delta(C)$ denote the number of nodes of $C$.

\begin{proposition}
  \label{prop:nodes}
  For a curve $C$ in $\M_{0,2g+2}$, if $[C] \in \mog{k}$ then
  $\delta(C) \leq g+k-1$.
\end{proposition}
\begin{proof}
  For a stable curve of genus $g$ which has $r$ irreducible
  components, $\delta$ nodes and the geometric genera of the
  irreducible components are $g_i$ we have the following equality:
  \begin{equation*}
    g = \sum_{i=1}^r (g_i-1) +\delta +1
  \end{equation*}
  This yields $\delta = (g-1) - \sum_{i=1}^r (g_i-1) $. Hence, if the
  curve has at most $k$ rational components, then $\delta \leq
  g+k-1 $. By Algorithm~\ref{Algo}, for a curve $D$ in $\M_{0,2g+2}$,
  $\pi([D])$ has at least as many nodes as $D$.
\end{proof}

\subsection{A Constructible Sheaf.}
Consider the constant sheaf $\C$ on $\M_{0,2g+2}$. Let $\L := \pi_*
\C$. (If the action of $S_{2g+2}$ on $\M_{0,2g+2}$ were free, $\L$
would be a local system, but in any case it is a constructible sheaf.)

Consider the restriction of $\L$ on $\hg{0}$ and denote it by $\L$ as
well. Then,
\begin{equation*}
  H^{\bullet}(\hg{0},\L) \cong H^{\bullet}(\mog{0}, \C)
\end{equation*}
By Poincar\'{e} Duality, Lemma~\ref{lem:main} is a corollary of the
following lemma, proved in Section~\ref{sec:cohomology}.
\begin{lemma}
  \label{lem:coh}
  The cohomology group $H^g_c(\mog{0})$ is non-zero, and
  $H^k_c(\mog{0}) =0$ for $k<g$.
\end{lemma}

\section{Cohomology Computations} \label{sec:cohomology}

\subsection{A Spectral Sequence.} \label{sec:ss} The spectral sequence
we describe here is ``dual'' to the spectral sequence of Deligne
\cite[section 1.4]{D} for mixed Hodge theory of smooth
quasi-projective varieties (see Getzler \cite[section 3.7]{Get1}).

Let $X$ be a smooth projective variety over $\C$ of dimension $n$, and
$D$, a simple normal crossings divisor. By that we mean $D =
D_1\cup\ldots\cup D_N$, where each $D_i$ is a co-dimension 1 smooth
sub-variety and all intersections of $D_i$ are transverse. Let
\begin{equation*}
  X =X_0\supset X_1\supset \ldots\supset X_n \supset X_{n+1}=
  \emptyset
\end{equation*}
be the following filtration on $X$: $X_1 = D$, and
\begin{equation*}
  X_k = \bigcup_{\substack{|I|=k \\I \subset \Set{1}{N}}} \bigcap_{i\in I} D_i
\end{equation*}
Let $X_k^\circ = X_k\setminus X_{k+1}$. Then we have
$H^{\bullet}(X_k,X_{k+1}) \cong H^{\bullet}_c(X_k^\circ)$, where
$H^{\bullet}$ denotes cohomology and $H^{\bullet}_c$ compactly
supported cohomology with complex coefficients.

Consider the spectral sequence associated to this filtration on $X$. We have 
\begin{equation*}
  E_1^{p,q} = H^{p+q}(X_{-p},X_{-p+1}) = H^{p+q}_c(X_{-p}^\circ)
\end{equation*}
and the differential $d_1$ is given by the composition of maps
\begin{equation*}
  \xymatrix{
    E_1^{p,q}\ar[rr]^{d_1}\ar@{=}[d] & & E_1^{p+1,q}\ar@{=}[d] \\
    H^{p+q}(X_{-p},X_{-p+1}) \ar[r]^-\imath & H^{p+q}(X_{-p})
    \ar[r]^-{\delta} & H^{p+q+1}(X_{-p-1},X_{-p})}
\end{equation*}
where $\imath$ and $\delta$ are the maps in the long exact sequence of
a pair $W \subset Z$ as follows:
\begin{equation*}
  \cdots\to H^{l-1}(W) \xrightarrow{\delta} H^l(Z,W)
  \xrightarrow{\imath} H^l(Z) \to \cdots
\end{equation*}
Since the filtration is finite the spectral sequence converges to
$H^{p+q}(X)$. Moreover, the vector spaces $E_1^{p,q}$ carry mixed Hodge
structures and the differential is a map of mixed Hodge
structures. The spectral sequence converges in the $E_2$ page and
$E_2^{p,q} = E_{\infty}^{p,q}$.

To apply this to our situation note that $X =\M_{0,m}$ is a smooth
projective complex variety and $D= \M_{0,m}\setminus M_{0,m}$ is a
simple normal crossings divisor. The set of (isomorphism classes of)
trees $\Gamma(0,m)$ can be partitioned into $\Gamma(0,m) =
\Gamma_0(0,n)\sqcup\ldots\sqcup\Gamma_{m-3}(0,m)$, where trees in
$\Gamma_k(0,m)$ have $k$ edges; then
\begin{equation*}
  X_k = \bigcup_{[T] \in \Gamma_k(0,m)} \M_T \quad \text{and} \quad
  X^\circ_k = \bigsqcup_{[T] \in \Gamma_k(0,m)} M_T .
\end{equation*}
As above we have a spectral sequence in the category of mixed
Hodge structures, with
\begin{equation}
  \label{eq:ss}
  _mE_1^{p,q} = \bigoplus_{[T] \in \Gamma_{-p}(0,m)}
  H^{p+q}_c(M_T)
\end{equation}
This spectral sequence tells us how to compute $H^{\bullet}(\M_{0,m})$
from the knowledge of $H^{\bullet}(M_{0,l})$ for all $l \leq m$.

The Hodge structure of $H^i(M_{0,l})$ is pure of weight $2i$ (Getzler
\cite{Get1}). Hence $H^i_c(M_{0,l})$ has a pure Hodge structure of
weight $2i-2(l-3)$ and $_mE_1^{p,q}$ carries a pure Hodge structure of
weight $2q -2(m-3)$. Since the cohomology $H^i(\M_{0,m})$ carries a
pure Hodge structure of weight $i$, and the odd cohomology is trivial,
we conclude that $H^{2i}(\M_{0,m})\cong {}_mE_2^{-(m-3-i),m-3+i}$ and
$_mE_2^{p,q}=0$ if $ q-p \neq 2(m-3)$. So we get a resolution of
$H^{2k}(\M_{0,m})$ as follows
\begin{equation*}
  0 \to H^{2k}(\M_{0,m}) \to \bigoplus_{[T] \in \Gamma_{m-3-k}(0,m)}
  H^{2k}_c(M_T) \to \ldots \to H^{m-3+k}_c(M_{0,m}) \to 0
\end{equation*}
Taking duals and setting $j = m-3-k$, we get 
\begin{equation}
  \label{eq:res}
  0 \to H_j(M_{0,m}) \to \bigoplus_{[T] \in \Gamma_{1}(0,m)} H_{j-1}
  (M_T) \to \ldots \to \bigoplus_{[T] \in \Gamma_{j}(0,m)} H_0 (M_T)
  \to H_{2j}(\M_{0,m}) \to 0
\end{equation}

\subsection{Truncation.}
Now $\mog{0} \subset \M_{0,2g+2}$, so if we set $\partial\mog{0} =
\M_{0,2g+2}\setminus \mog{0}$, then
\begin{equation*}
  \quad H^{\bullet}_c(\mog{0}) \cong H^{\bullet}(\M_{0,2g+2},\dmog)
\end{equation*}
Let $C$ be a stable genus zero curve and $T$ its dual graph. From
Proposition~\ref{prop:admissible} we have $[C] \in \mog{0}$ if and
only if $\rho(v)>2$ for all internal vertices $v \in V(T)$, or, since
$\rho(v)$ is even (by Riemann-Hurwitz), $\rho(v) \geq 4$. Let us call
trees satisfying this condition \textbf{good} trees and denote the set
of isomorphism classes of good trees of type $(0,2k)$ by
$\Gamma(0,2k)^{0}$; then
\begin{equation*}
  \mog{0} = \bigsqcup_{[T] \in \Gamma(0,2g+2)^0} M_T .
\end{equation*}
Let $\Gamma_k(0,2g+2)^0$ be the isomorphism classes of good trees with
$k$ edges, so that
\begin{equation*}
  \Gamma(0,2g+2)^0 = \Gamma_0(0,2g+2)^0 \sqcup\ldots\sqcup
  \Gamma_{g-1}(0,2g+2)^0
\end{equation*}
The filtration on $\mog{0}$ gives a filtration on the singular
co-chains of the pair $(\M_{0,2g+2},\dmog)$, and we have the
associated spectral sequence:
\begin{equation*}
  {}_gF_1^{p,q} = \bigoplus_{[T] \in \Gamma_{-p}(0,2g+2)^0}
  H^{p+q}_c(M_T)
\end{equation*}
The differential here is the same as the differential of the previous 
spectral sequence.

We have the bounds
\begin{equation}
  \label{eq:bds}
  \text{${}_gF_1^{p,q} =0$ unless $1-g\le p\le 0$ and $2g-1\le q\le 4g-2$.}
\end{equation}
To see this, first note that by Proposition \ref{prop:nodes}, a good graph 
can have at most $g-1$ edges. This gives the bounds on $p$. Further when
$T$ has $r$ edges, $M_T$, is an affine variety of dimension $2g-1-r$. Hence, 
the compactly supported cohomology of $M_T$ is non-trivial in degrees 
$2g-1-r$, through $4g-2-2r$. This shows the bound on $q$.

Again the spectral sequence converges and $H^k_c(\mog{0}) \cong
\bigoplus_s {}_gF^{-s,k+s}_\infty$. As before, this is a spectral
sequence in the category of mixed Hodge structures.

From the above bounds \eqref{eq:bds}, it is clear that 
\begin{equation*}
  H^g_c(\mog{0}) \cong {}_gF_{\infty}^{-g+1, 2g-1} \cong
  {}_gF_2^{-g+1,2g-1} .
\end{equation*}
Also ${}_gF_2^{-g+1,2g-1}$ has a pure Hodge structure of weight $0$.

Moreover $H^k_c(\mog{0})=0$ for $k<g$ since ${}_gF_1^{p,q} =0$ if $q+p
< g$. Hence, to complete the proof of Lemma~\ref{lem:coh}, we just need
to show ${}_gF_2^{-g+1,2g-1} \neq 0$.

\subsection{A digression into Operads.}
Here we borrow notations and definitions from \cite{Get2}. Recall that
an $\S$-module $\V$ is a sequence of chain complexes
\begin{equation*}
  \{\V(n)\mid n\ge 0 \}
\end{equation*}
together with an action of $S_n$ on $\V(n)$. 

If $V$ is a chain complex, let $\Sigma V$ be its shift (sometimes
denoted $V[1]$). The gravity and hypercommutative operads \cite{Get2}
have as their underlying $\S$-modules
\begin{align*}
  \Grav(n) &=
  \begin{cases}
    \Sigma^{2-n}\sgn_n \otimes H_{\bullet}(M_{0,n+1}), & n\geq 2 \\
    0, & n<2 \\
  \end{cases} \intertext{and}
  \Hycomm(n) &=
  \begin{cases}
    H_{\bullet}(\M_{{0,n+1}}), & n\geq 2 \\
    0, & n<2
  \end{cases} .
\end{align*}

For an $\S$-module $\V$, the dual $\S$-module $\V^{\vee}$
is defined as
\begin{equation*}
  \V^{\vee}(n) = \Sigma^{n-2}\sgn_n\otimes \V(n)^{*} \, .
\end{equation*}
The double dual $(\V{}^\vee){}^\vee$ of an $\S$-module is naturally
isomorphic to $\V$.

Let
\begin{equation*}
  \V(n) = H_c^{\bullet} (M_{0,n})\, .
\end{equation*}
By Poincar\'{e} duality, $\Grav(n) \cong \V^{\vee}(n)$. So after
taking duals, we have $\V \cong \Grav^{\vee}$.

Summing the complexes
\begin{equation*}
  0 \to H_j(M_{0,m}) \to \bigoplus_{[T] \in \Gamma_{1}(0,m)} H_{j-1}
  (M_T) \to \ldots \to \bigoplus_{[T] \in \Gamma_{j}(0,m)} H_0 (M_T)
  \to 0
\end{equation*}
of \eqref{eq:res} placed in degrees $[j,2j]$, we get an $\S$-module
\begin{equation*}
  \mathcal{W}(n) =  \bigoplus_{[T] \in \Gamma(0,n)} H_{\bullet} (M_T) .
\end{equation*}
The cohomology of $\mathcal{W}(n)$ is isomorphic to $\Hycomm(n)$; 
this is just a restatement of the Koszul duality of $\Grav$ and $\Hycomm$, 
since $\mathcal{W}$ is the cobar construction for $\Grav$ (see \cite{Get2}).

A diagram chase shows that the differential $d_1$ in the spectral
sequence \eqref{eq:ss} is adjoint to the differential in the cobar
construction for $\Grav$.

\subsection{Proof of Lemma \ref{lem:coh}.}
As we already noted
\begin{equation*}
  H^g_c(\mog{0}) \cong {}_gF_{\infty}^{-g+1, 2g-1} \cong
  {}_gF_2^{-g+1,2g-1},
\end{equation*}
so the strategy of proof will be to show that $d_1 :
{}_gF_1^{-g+1,2g-1} \to {}_gF_1^{-g+2,2g-1}$ has a kernel. The
spectral sequence ${}_gF^{\bullet,\bullet}$ is a truncation of the
spectral sequence ${}_{2g+2}E^{\bullet,\bullet}$, as in
section~\ref{sec:ss}, and has the same differential in the first
page. We identify a subspace $V_{g,g}$ of ${}_{2g+2}E_1^{-g,2g-1}$ on
which the differential $d_1$ is non-trivial and show that the image is
inside ${}_gF_1^{-g+1,2g-1} \subset {}_{2g+2}E_1^{-g+1,2g-1}$. First
we define some specific trees which will be useful in the following
discussion.

For each $l$, $0 \leq l \leq g$ consider the tree $T_{l,g}$ of type
$(0,2g+2)$ defined as follows (see Figure~\ref{fig:graph}):
\begin{enumerate}
\item $T_{l,g}$ has vertices $v_0,v_1, \ldots , v_l$;
\item vertex $v_0$ has $2g-2l+2$ leaves and $v_i$ has 2 leaves for each
 $i>0$;
\item $v_0$ is connected to each $v_i$ for $i>0$  by an edge;
\item for $i>0$, the leaves of $v_i$ are numbered $2i-1,2i$ and the leaves of
  $v_0$ are $2l+1,\ldots,2g+2$. 
\end{enumerate}

\begin{figure}[h]
\centering
 \includegraphics[scale=1]{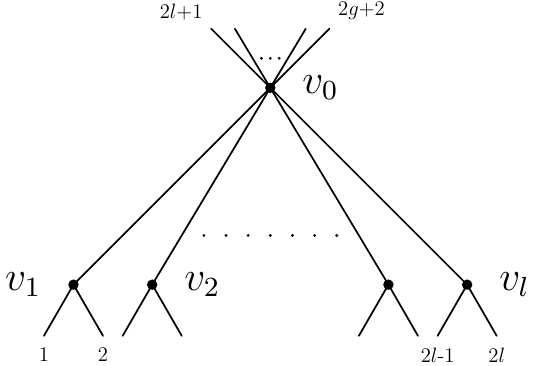}
\caption{The tree $T_{l,g}$} 
\label{fig:graph}
\end{figure}

Note that the stratum $M_{T_{l,g}}$ is isomorphic to $M_{0,2g-l+2}$ and 
has dimension $2g-l-1$. One should think of $M_{T_{l,g}}$ as the moduli 
space $M_{0,2g-l+2}$ with two sets of marked points, $l$ of them even and 
the rest odd. Let
\begin{equation*} 
  W_{l,g} =  H^{2g-1-l}_c(M_{T_{l,g}}) \qquad \text{for $l=0,
    \ldots,g$.}
\end{equation*}  

When $g=4$, and $l=2$, Figure~\ref{fig:curve} shows a curve with dual 
graph $T_{l,g}$, the admissible cover and its stabilization.
\begin{figure}[h]
\centering
 \includegraphics[scale=0.8]{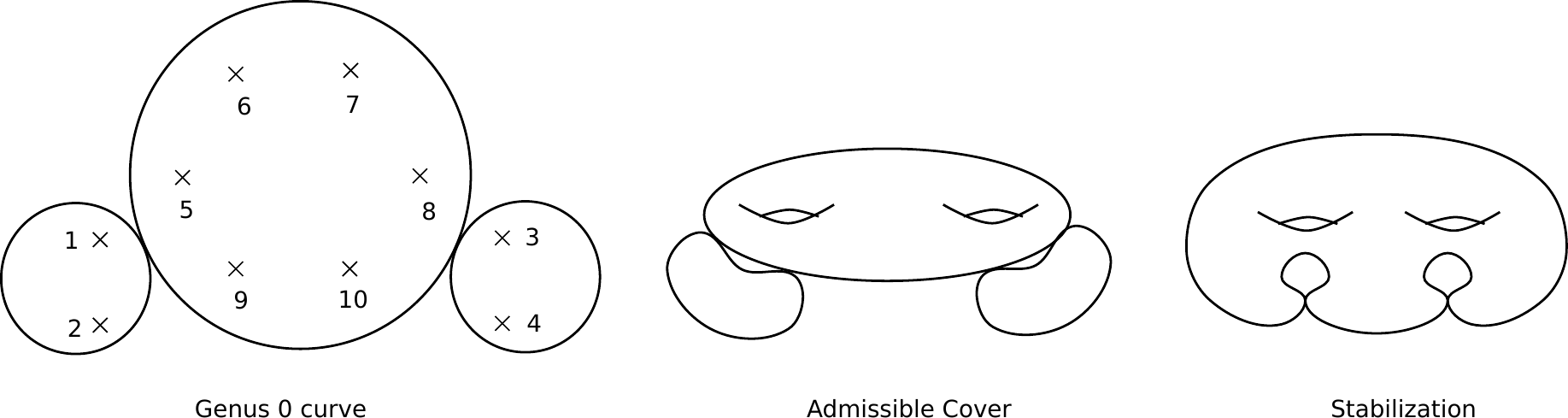}
\caption{A curve with dual graph $T_{2,4}$}  
\label{fig:curve}
\end{figure} 

The symmetric group $S_n$ acts on $M_{0,n+1}$ by permuting the first
$n$ marked points, and hence on $H^{\bullet}(M_{0,n+1})$.  We treat
$T_{l,g}$ as a rooted tree with the leaf $2g+2$ as the root. Then
\begin{equation*}
  \Aut(T_{l,g}) \cong S_{2g-2l+1}\times(S_l\wr S_2) \subset S_{2g+1}
\end{equation*}

$\Aut(T_{l,g})$ acts on $W_{l,g}$ and we have an induced
representation of $S_{2g+1}$
\begin{equation*}
  \Ind_{\Aut(T_{l,g})}^{S_{2g+1}} W_{l,g} .
\end{equation*}
(This corresponds to summing over the appropriate cohomology of the
strata corresponding to the rooted trees that are isomorphic to
$T_{l,g}$ after renumbering of the non-root leaves.)

\begin{definition} For $0\leq l \leq g$, we define the vector space 
$V_{l,g}$ to be the subspace of $W_{l,g}$ of invariants under the action 
of $\Aut(T_{l,g}) $
\begin{equation*} 
  V_{l,g} = \left( W_{l,g} \right)^{\mathrm{Aut}(T_{l,g})}\cong 
  \left( \Ind_{\Aut(T_{l,g})}^{S_{2g+1}} W_{l,g} \right)^{S_{2g+1}} .
\end{equation*}
\end{definition}

Recall the spectral sequences ${}_mE^{\bullet,\bullet}$ corresponding
to the cohomology of $\M_{0,m}$ and ${}_gF^{\bullet,\bullet}$
corresponding to the compactly supported cohomology of $\mog{0}$. Then
$V_{l,g} \subset {}_{2g+2}E_1^{-l,2g-1} $ and when $l\leq g-1$, we have
$V_{l,g} \subset {}_gF_1^{-l,2g-1}$.

We have the following diagram
\begin{equation}
  \label{comm-diag}
  \begin{aligned}
    \xymatrix@M=5pt{ _{2g+2}E^{-g,2g-1}_1 \ar[r]^-{d_1} &
      _{2g+2}E^{-g+1,2g-1}_1
      \ar[r]^-{d_1} & _{2g+2}E^{-g+2,2g-1}_1 \\
      V_{g,g} \ar@{^{(}->}[u] \ar[r]^{d_1} & V_{g-1,g} \ar@{^{(}->}[u]
      \ar[r]^{d_1} \ar@{_{(}->}[d] & V_{g-2,g}
      \ar@{^{(}->}[u] \ar@{_{(}->}[d] \\
      & _gF^{-g+1,2g-1}_1 \ar[r]^{d_1} & _gF^{-g+2,2g-1}_1\\
    }
  \end{aligned}
\end{equation}

To understand the vector spaces $V_{l,g}$, let us first analyse $W_{l,g}$. 
As shown in \cite{Get2}, as representations of $S_n$ we have
\begin{align}
  \label{eq:lie1}
  H^{n-2}_c(M_{0,n+1}) &\cong H_{n-2}(M_{0,n+1}) \\
  &\cong \sgn_n \otimes \Lie(n) . \notag
\end{align}

\begin{definition}
  A \textbf{Lie superalgebra} is a $\mathbb{Z}_2$-graded vector space
  $L = L_0 \oplus L_1$ along with a bracket $[\bullet, \bullet]$,
  which satisfies the following axioms: if $a,b,c \in L$ are
  homogeneous elements of degree $|a|, |b|$ and $|c|$, then
  \begin{enumerate}
    \item $[a,b] \in L_{|a|+|b|}$;
    \item $[a,b] = (-1)^{|a||b|}[b,a]$;
    \item $(-1)^{|a||c|}[a,[b,c]] + (-1)^{|c\|b|}[c,[a,b]]
          + (-1)^{|b||a|}[b,[c,a]] = 0$.
  \end{enumerate}
  A $\Z$-graded Lie algebra is defined in the same way, except that the 
  vector space has a $\Z$-grading.
\end{definition}

The $S_n$-module $\Lie(n)$ associated to the operad $\Lie$ is a submodule 
of the free Lie algebra with generators $\{x_1,\dots,x_n\}$;
similarly, the $S_n$-module
\begin{equation} \label{eq:lie2}
  \Lambda\Lie(n)=\sgn_n\otimes\Lie(n)[1-n]
\end{equation}
associated to the operad $\Lambda\Lie$ (suspension of the operad
$\Lie$) is a submodule of the free algebra with a shifted Lie
bracket. Note that this is $\Z$-graded, but we can consider the
underlying $\Z_2$-grading. This turns out to be a submodule of the
free Lie superalgebra with generators $\{y_1,\dots,y_n\}$ of
degree~$1$.

Let $A = A_0 \sqcup A_1$ (disjoint union) be a $\mathbb{Z}_2$-graded set,
and $A^*$, the free monoid generated by $A$. Denote by $|a|$ the degree 
of $a \in A^*$. Then $\mathbb{C} \langle A^* \rangle$, the $\mathbb{C}$  
vector space generated by $A^*$ with the obvious multiplication, is 
called the free nonassociative algebra over $A$. Define $[a,b] = 
ab-(-1)^{|a||b|}ba$. Let $I$ be the ideal in $\mathbb{C} \langle A^* 
\rangle$ generated by the set 
\begin{equation*}
  \{ ab+(-1)^{|a||b|}ba , (-1)^{|a||c|}[a,[b,c]] + 
  (-1)^{|c||b|}[c,[a,b]] + (-1)^{|b||a|}[b,[c,a]] \mid
  a,b,c \in A^*\} \, .
\end{equation*}
Then $L_A = \mathbb{C} \langle A^*\rangle / I$ with the binary operation 
$[\bullet,\bullet]$ is a Lie superalgebra called the free Lie superalgebra 
with generators $A$.

From \eqref{eq:lie1} and \eqref{eq:lie2}, it is clear that
\begin{equation}
  W_{l,g} \cong \Lambda\Lie(T_{l,g}) .
\end{equation}
In other words, $W_{l,g}$ is spanned by free Lie superalgebra words in
generators
\begin{equation*}
  \{[a_1,a_2],[a_3,a_4],\ldots,[a_{2l-1},a_{2l}],a_{2l+1},
  \ldots,a_{2g+1}\}
\end{equation*} 
where $a_i$ has degree $1$, and in which each letter $a_i$ occurs
exactly once. This vector space is isomorphic to the vector space
spanned by free Lie superalgebra words in generators
\begin{equation*}
	\{b_1,\ldots,b_l,a_{2l+1},\ldots,a_{2g+1} \}
\end{equation*}
where again each generator occurs once, but now $b_i$ has degree $0$
whereas $a_j$ has degree $1$.

Let $A$ be a $\mathbb{Z}_2$-graded ordered alphabet and $A^*$ the free 
monoid generated by $A$ ordered lexicographically. A word $w$ is a 
\textbf{Lyndon} word if it is lexicographically smaller than all its 
cyclic rearrangements. In other words for any non-trivial factorization 
$w = uv$, we have $w < v$. 

To a Lyndon word over $A$ one can uniquely associate an element of the
free Lie superalgebra generated by $A$. This association is called the  
standard bracketing of a Lyndon word and is defined inductively on the 
length of the word. We denote the bracket of a Lyndon word $w$ by $B(w)$. 

Suppose $w = uv$ where $v$ is the lexicographically smallest proper right 
factor of $w$. Then $u$ and $v$ are both Lyndon words and $B(w) = [B(u),
B(v)]$. Let $\mathfrak{L}(A)$ be the set of all Lyndon words on the alphabet 
$A$. The alphabet $A$ can have elements in different degrees and we define 
the degree of $w \in A^*$ by 
\begin{equation*}
  |w| = \sum_{i=1}^k |a_i| , \quad \text{if $w = a_1\cdots a_k$.}
\end{equation*}

The set
\begin{equation*}
  \{ B(w) \mid w\in \mathfrak{L}(A) \} \cup \{ [B(w),B(w)] \mid
  \text{$w \in \mathfrak{L}(A)$, where $|w|=1$} \}
\end{equation*} 
forms a basis of the free Lie superalgebra with generators $A$ called
the Lyndon basis; see for example Shtern \cite{Sh} (also Reutenauer 
\cite[Sections 4.1 and 5.1]{Reu}).

Clearly, the $V_{l,g}$ are in one-to-one correspondence with the
$S_l\times S_{2g-2l+1}$ invariants of $W_{l,g}$, which acts by 
permuting the letters $\{b_1,\ldots,b_l\}$ and
$\{a_{2l+1},\ldots,a_{2g+1}\}$ separately. This proves the first part
of the following lemma.

\begin{lemma}
  Let $\Lie_{(i,j)}[a,b]$ denote the vector space of Lie superalgebra
  words in the letters $\{a,b\}$, where $a$ has degree $1$ and $b$ has
  degree $0$, homogeneous of degree $i$ in $a$ and $j$ in $b$. Then
  $V_{l,g}\cong\Lie_{(2g-2l+1,l)}[a,b]$. Furthermore, $\dim V_{g,g} = 1$.
\end{lemma}
\begin{proof} 
  Let us define the order $a<b$ for the generators $\{a,b\}$. Then
  the only Lyndon word with one instance of $a$ and $g$ instances of
  $b$ is $ab^g$; this shows that
  $V_{g,g}$ has dimension $1$, with basis $B(ab^g)=[\dots[a,b],\dots,b]$.
\end{proof}

The following lemma is the main ingredient in the proof of
Lemma~\ref{lem:coh}.
\begin{lemma} 
  The differential $d_1: V_{g,g} \to V_{g-1,g}$ is non-zero.
\end{lemma}

\begin{proof}
  The differential $d_1$ in the complex \eqref{eq:res} is the 
  adjoint of the differential of the cobar construction for the 
  gravity operad. Hence the differential corresponds to an 
  alternating sum of operadic compositions in the gravity operad.  

  Since the vector space $V_{g,g}$ is one-dimensional with basis
  vector $\omega_g = B(ab^g)$, the differential $d_1: V_{g,g} \to
  V_{g-1,g}$ is determined by its action on $\omega_g$, which is an
  alternating sum of the terms obtained by replacing each instance of
  $b$ by $[a,a]$. In other words,
  \begin{equation} 
    \label{diff}
    d_1(\omega_g) = \sum_{i=1}^{g} (-1)^{i-1}
    [\dots[[\omega_{i-1},[a,a]],b],\dots,b] .
  \end{equation} 

  Let $\mathfrak{L}(3,g-1)$ be the set of Lyndon words in $\{a,b\}$
  with $3$ instances of $a$ and $(g-1)$ of $b$. We claim that for
  $g\geq 2$,
  \begin{equation*}
    d_1(\omega_g) = 
    \begin{cases}
      \displaystyle 2\,B(a^3b^{g-1}) + (g-2)B(a^2bab^{g-2}) +
      \sum_{\substack{w \in \mathfrak{L}(3,g-1) \\ w > a^2bab^{g-2}}}
      n_wB(w) , & \text{$g$ even,} \\ \\
      \displaystyle (g-1)B(a^2bab^{g-2}) + \sum_{\substack{w \in
          \mathfrak{L}(3,g-1) \\ w > a^2bab^{g-2}}} n_w B(w) , &
      \text{$g$ odd.}
    \end{cases}
  \end{equation*}
  The cases $g=2,3$ are true: by the super-Jacobi identity,
  $[a,[a,a]]=0$, hence
  \begin{align*}
    d_1(\omega_2) &= d_1[[a,b],b] = [[a,[a,a]],b] - [[a,b],[a,a]] \\
    &= [[a,a],[a,b]] = 2[a,[a,[a,b]]] = 2\,B(aaab) \\
    d_1(\omega_3) &= d_1[[[a,b],b],b] = [[[a,[a,a]],b],b] -
    [[[a,b],[a,a]],b] + [[[a,b],b],[a,a]] \\
    &= 2\,[[a,[a,b]],[a,b]] = 2\,B(aabab) .
  \end{align*}
  The induction now follows on combining the following results:
  \begin{enumerate}
  \item By the standard triangularity property for the Lyndon basis,
    if $m<n$ are Lyndon words, and thus $mn$ is again a Lyndon word,
    then
    \begin{equation*}
      [B(m),B(n)] = B(mn) + \sum_{\substack{ w 
          \text{ Lyndon word} \\ |w| = |m|+|n|,\ w > mn}} n_w B(w) .
    \end{equation*}
  \item Expanding $[B(a^3b^{g-2}),b]$, and applying the
    super-Jacobi identity, we see that if $g>3$, then
    \begin{equation*}
      [B(a^3b^{g-2}),b] = B(a^3b^{g-1}) + B(a^2bab^{g-2}) - B(a^2b^{g-2}ab) .
    \end{equation*}
  \item From \eqref{diff}, it follows that
    \begin{align*}
      d_1(\omega_g) &= [d_1(\omega_{g-1}),b] +
      (-1)^{g-1}[\omega_{g-1},[a,a]] \\
      &= [d_1(\omega_{g-1}),b] + (-1)^g\,2\,B(a^3b^{g-1}) .
    \end{align*}
  \end{enumerate}
  The lemma is proved.
\end{proof}

The image of $d_1: V_{g,g} \to V_{g-1,g}$ lies in the kernel of $d_1:
V_{g-1,g} \to V_{g-2,g}$.  Since $V_{g-1,g} \subset
{}_gF_1^{-g+1,2g-1}$, this implies that $H^g_c(\mog{0}) \cong
{}_gF_2^{-g+1,2g-1}$ is non-trivial (see \eqref{comm-diag}),
completing the proof of Lemma~\ref{lem:coh}.

\end{document}